\documentclass[11pt]{amsart}

\usepackage{epsf,amssymb,times,overpic,amscd}
\usepackage[usenames]{color}
\usepackage[all]{xy}

\usepackage{amsfonts}
\usepackage{amsmath}
\usepackage{graphicx}
\usepackage{epsfig}
\usepackage{epstopdf}

\newcommand{\ba}{\begin{array}}
\newcommand{\ea}{\end{array}}
\newcommand{\be}{\begin{enumerate}}
\newcommand{\ee}{\end{enumerate}}

\newtheorem{thm}{Theorem}[section]
\newtheorem{prop}[thm]{Proposition}
\newtheorem{lemma}[thm]{Lemma}

\theoremstyle{definition}
\newtheorem{defn}[thm]{Definition}

\theoremstyle{remark}

\newtheorem{rmk}[thm]{Remark}

\newcommand{\R}{\mathbb{R}}
\newcommand{\Z}{\mathbb{Z}}

\newcommand{\Hom}{\operatorname{Hom}}
\newcommand{\E}{\mathcal{E}}
\newcommand{\ci}{\mathcal{I}}
\newcommand{\n}{\noindent}
\newcommand{\im}{\sqrt{-1}}
\newcommand{\F}{\mathbb{F}_2}
\newcommand{\ot}{\otimes}
\newcommand{\ra}{\rightarrow}

\newcommand{\mf}{\mathbf}
\newcommand{\op}{\operatorname}
\newcommand{\cal}{\mathcal}

\newcommand{\lan}{\langle}
\newcommand{\ran}{\rangle}

\begin{document}
\title{A categorification of the square root of $-1$}

\author{Yin Tian}
\address{Simons Center for Geometry and Physics,
State University of New York,
Stony Brook, NY 11794}
\email{ytian@scgp.stonybrook.edu}

\begin{abstract}
We give a graphical calculus for a monoidal DG category $\cal{I}$ whose Grothendieck group is isomorphic to the ring $\Z[\im]$.
We construct a categorical action of $\cal{I}$ which lifts the action of $\Z[\im]$ on $\Z^2$.
\end{abstract}

\maketitle

\section{Introduction}
Categorification is a procedure which lifts operations on the level of sets or vector spaces to those on the level of categories.
For instance, the addition could be lifted to the operation of direct sum in an additive category; the multiplication with $-1$ could be upgraded to a shift functor in a triangulated category.
The square root $\im$ is another fundamental concept in mathematics.
The goal of this paper is to give a naive diagrammatic categorification of the ring $\Z[\im]$.

The program of categorification was initialized by Crane and Frenkel \cite{CF} in a construction of four-dimensional topological quantum field theory.
A celebrated example is Khovanov homology \cite{Kh1} of links in $S^3$ whose graded Euler characteristic agrees with the Jones polynomial.
Categorification at roots of unity could be the first step towards categorifying quantum invariants of $3$-manifolds.
In particular, Khovanov \cite{Kh2} categorified a prime root of unit using Hopfological algebras.
Later on the small quantum $\mathfrak{sl}(2)$ at a prime root of unit was categorified by Khovanov and Qi \cite{KhQ}.
In this context, this paper can be viewed as an attempt to categorify the simplest non-prime root of unit $\im$.
The construction of our categorification is diagrammatic.
The diagrammatic approach was pioneered by Lauda in his categorification of the quantum group $\mathbf{U}_q(\mathfrak{sl}_2)$ \cite{La}.


Our main results are the followings.
Consider an action of $\Z[\im]$ on $\Z^2=\Z\lan x ,y \ran$ where $\im \cdot x =y, \im \cdot y = -x$.
Let $H(\cal{A})$ denote the $0$th cohomology category of a DG category $\cal{A}$.
Let $K_0(\cal{A})$ denote the Grothendieck group of $H(\cal{A})$ if $H(\cal{A})$ is triangulated.

\begin{thm} \label{thm1}
There exists a DG category $\cal{I}$ such that $H(\cal{I})$ is triangulated.
There is a monoidal bifunctor $\ot: \ci \times \ci \ra \ci$ whose decategorification makes $K_0(\ci)$ isomorphic to the ring $\Z[\im]$.
\end{thm}

\begin{thm} \label{thm2}
There exists a DG category $DGP(R)$ generated by some projective DG modules over a DG algebra $R$ such that $H(DGP(R))$ is triangulated.
The Grothendieck group $K_0(DGP(R))$ is isomorphic to $\Z^2$.
Moreover, there is a categorical action
$\eta: \ci \times DGP(R) \ra DGP(R)$ which lifts the action of $\Z[\im]$ on $\Z^2$.
\end{thm}

The DG algebra $R$ is a quotient of a quiver algebra of the Khovanov-Seidel quiver \cite{KhS} with two vertices.
The differential on $R$ is trivial.
Recall that the quiver algebra $R$ is motivated from the Fukaya category $\cal{F}(T^2)$ of the torus.
In particular, the category $DGP(R)$ is related to a subcategory of $\cal{F}(T^2)$ which is generated by two Lagrangians of slopes $0$ and $\infty$ on $T^2$.
The automorphism $Q$ of $DGP(R)$ which lifts the multiplication with $\im$ on $K_0(DGP(R))$ can be visualized as a rotation of $\frac{\pi}{2}$ around the intersection point of the two Lagrangians.

A natural question for us is to find other categories which admit categorical actions of $\ci$.

\vspace{.2cm}
The organization of the paper is the following.
In Section 2 we define the monoidal DG category $\ci$ and show that there is a surjective ring homomorphism $\Z[\im] \twoheadrightarrow K_0(\ci)$.
In Section 3 we define the DG category $DGP(R)$ and show that $K_0(DGP(R))$ is isomorphic to $\Z^2$.
Then we construct the categorical action of $\ci$ on $DGP(R)$ via DG $R$-bimodules.

\vspace{.2cm}
\noindent {\bf Acknowledgements:}
I would like to thank Mikhail Khovanov for many helpful conversations.

\section{The monoidal DG category $\ci$}
We define $\ci$ in two steps.
In Section 2.1 we construct an additive monoidal DG category $\ci'$ using diagrams.
In Section 2.2 we enlarge $\ci'$ to the category $\ci$ of {\em one-sided twisted complexes} over $\ci'$.
Then we discuss the Grothendieck group $K_0(\ci)$ and prove the following.

\begin{prop}\label{K0ci}
There exists a monoidal DG category $\ci$ such that $H(\ci)$ is triangulated.
Moreover, there is a surjective ring homomorphism $\gamma: \Z[\im] \twoheadrightarrow K_0(\ci)$.
\end{prop}
We will show that $\gamma$ is actually an isomorphism via the categorical action of $\ci$ in Section 3.

\subsection{The category $\ci'$}
In this subsection we define the additive DG category $\ci'$ whose morphism sets are cochain complexes of $\F$-vector spaces with trivial differential.
In other words, the morphism sets are just $\Z$-graded $\F$-vector spaces.
We fix $\F$ as the ground field for morphisms throughout the paper.
We expect some difficulties in keeping track of signs when generalizing the ground field from $\F$ to $\Z$.

\vspace{.2cm}
\n $\bullet$ {\bf Objects:} the set of elementary objects $\E(\ci')$ of $\ci'$ consists of nonnegative products $\{Q^k; k\geq 0\}$ of a single generator $Q$.
The unit object $\mf{1}$ corresponds to $Q^0$.
Let $Q^k[m]$ denote an object $Q^k$ with a cohomological grading shifted by $m \in \Z$.
In general, an object of $\ci'$ is a formal direct sum $\bigoplus_{i=1}^{n}Q^{k_i}[m_i]$.

\vspace{.2cm}
\n $\bullet$ {\bf Morphisms:} a morphism set $\Hom_{\ci'}(Q^k, Q^l)$ is a $\Z$-graded $\F$-vector spaces generated by a set $D(k,l)$ of planar diagrams from $k$ points to $l$ points, modulo local relations.
Morphisms extend linearly to formal direct sums.
The composition of morphisms is given by stacking diagrams vertically.
A vertical stacking of two diagrams is defined to be zero if their endpoints do not match.
The monoidal functor on the elementary objects and morphisms is given by the horizontal composition.
Throughout this subsection, we use $\Hom$ to denote $\Hom_{\ci'}$ for simplicity.

\vspace{.2cm}
\n $\bullet$ {\bf Diagrams:} any diagram $f \in D(k,l)$ is obtained by vertically stacking finitely many {\em generating diagrams} in the strip $\R \times [0,1]$ such that $f \cap (\R \times \{0\})=\{1, \dots, k\} \times \{0\}$ and $f \cap (\R \times \{1\})=\{1, \dots, l\} \times \{1\}$.
All diagrams are read from bottom to top as morphisms.
Each generating diagram is a horizontal composition of one {\em elementary diagram} with some trivial vertical strands.
The elementary diagrams consist of $4$ types as shown in Figure \ref{1}:
\be
\item a vertical strand $id_Q \in \Hom(Q,Q)$;
\item a cup $cup \in \Hom(Q^0, Q^2)$;
\item a cap $cap \in \Hom(Q^2, Q^0)$;
\item a half strand $hf \in \Hom(Q^0, Q)$.
\ee
\begin{figure}[h]
\begin{overpic}
[scale=0.23]{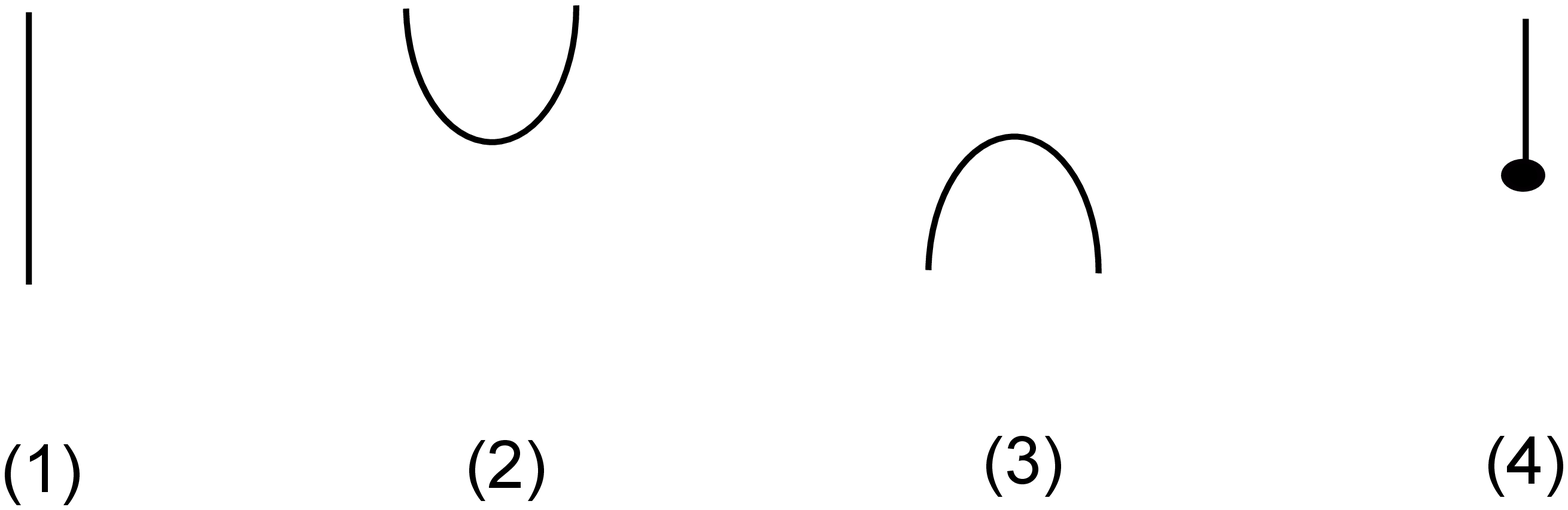}
\end{overpic}
\caption{The elementary diagrams of $\ci'$}
\label{1}
\end{figure}

\n $\bullet$ {\bf Local relations:}
the relations consist of 4 groups:

\vspace{.1cm}
\n {\bf (R1)} Isotopy relation:
\begin{description}
\item[(R1-a)] vertical strands as idempotents;
\item[(R1-b)] isotopy of a single strand;
\item[(R1-c)] isotopy of disjoint diagrams.
\end{description}

\vspace{.1cm}
\n {\bf (R2)} Handle slides relation.

\vspace{.1cm}
\n {\bf (R3)} Loop relation: a loop equals the empty diagram, i.e., the identity $id_{Q^0} \in \Hom(Q^0, Q^0)$.

\vspace{.1cm}
\n {\bf (R2)} Commutativity of a half strand: a half strand commutes with a trivial strand.
\begin{figure}[h]
\begin{overpic}
[scale=0.22]{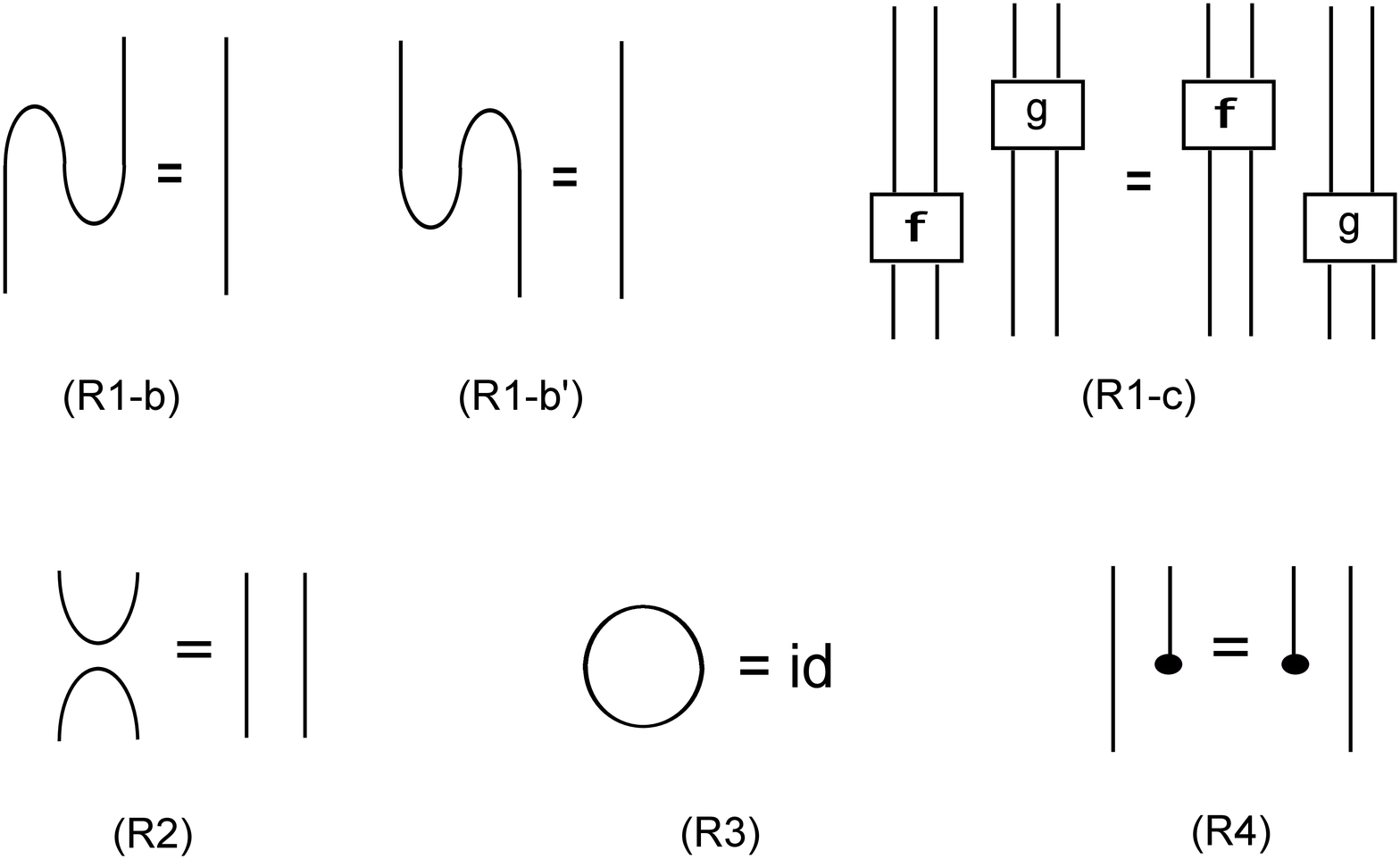}
\end{overpic}
\caption{Local relations.}
\label{2}
\end{figure}

\vspace{.2cm}
\n $\bullet$ {\bf Cohomological grading:}
A cohomological grading $\op{gr}$ is defined on the elementary diagrams by:
$$\op{gr}(id_Q)=\op{gr}(hf)=0, \quad \op{gr}(cap)=1, \quad \op{gr}(cup)=-1.$$
Two sides of any relation have the same grading.
Hence $\op{gr}$ defines a grading on the morphism sets.
Let $\Hom(Q^k, Q^l)=\bigoplus_i\Hom^i(Q^k, Q^l)$ be the decomposition according to the grading.

\begin{rmk}
The morphisms are actually generated by diagrams up to isotopy relative to boundary by (R1).
In particular, the object $Q[-1]$ is bi-adjoint to $Q$.
\end{rmk}

\n $\bullet$ {\bf More relations:} we deduce more relations from the defining relations.
\begin{lemma} \label{morerel}
(1) The relations (R1-b) and (R1-b') are equivalent.

\n (2) a cup or a cap commutes with a trivial strand.
\end{lemma}

\begin{proof}
We prove that (R1-b) implies (R1-b') and a cup commutes with a trivial strand as in Figure \ref{3}.
The verification of the other relations is similar.
\begin{figure}[h]
\begin{overpic}
[scale=0.25]{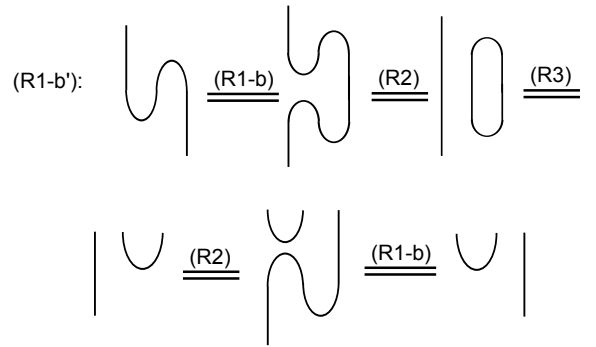}
\end{overpic}
\caption{More relations}
\label{3}
\end{figure}
\end{proof}

\vspace{.2cm}
\n $\bullet$ {\bf The key isomorphism:}
Both morphisms $cup \in \Hom(Q^0[1], Q^2)$ and $cap \in \Hom(Q^2, Q^0[1])$ have cohomological grading $0$.
Moreover, they are inverse to each other by (R2) and (R3). So we have the following.
\begin{lemma} \label{isoQ}
The isomorphism $Q^2 \cong Q^0[1]$ holds in $\ci'$.
\end{lemma}

\n $\bullet$ {\bf Bases of Hom spaces:} it suffices to give a basis for $\Hom(Q^k, Q^l)$ for $k,l \in \{0,1\}$ due to the isomorphism $Q^2 \cong Q^0[1]$.
We introduce more notations here.
Let $\overline{hf} \in \Hom(Q, Q^0)$ to denote the opposite half strand as in Figure \ref{4}.
Note that $\op{gr}(\overline{hf})=1$.
Let $\alpha_0= \overline{hf} \circ hf \in \op{End}(Q^0)$ denote a vertical composition of two half strands, and $\alpha_1 \in \op{End}(Q)$ be a horizontal composition of $\alpha_0$ and a trivial strand.
\begin{figure}[h]
\begin{overpic}
[scale=0.25]{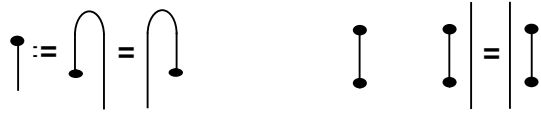}
\end{overpic}
\caption{$\overline{hf}$ is on the left and $\alpha_0, \alpha_1$ are on the right.}
\label{4}
\end{figure}

Since any nontrivial diagram commutes with a trivial strand, it is straightforward to get the following $\F$-bases of the Hom spaces.
\begin{lemma} \label{bases}
The Hom spaces $\Hom(Q^k, Q^l)$ for $k,l \in \{0,1\}$ are given by:

\n (1) $\op{End}(Q^i)=\lan \alpha_i^n ~|~ n \geq 0\ran$ for $i=0,1$;

\n (2) $\Hom(Q^0, Q)=\lan hf \circ \alpha_0^n ~|~ n \geq 0\ran=\lan \alpha_1^n \circ hf ~|~ n \geq 0 \ran$;

\n (3) $\Hom(Q, Q^0)=\lan \overline{hf} \circ \alpha_1^n ~|~ n \geq 0\ran=\lan \alpha_0^n \circ \overline{hf} ~|~ n \geq 0 \ran$.
\end{lemma}

In other words, $\op{End}(Q^i)$ is the polynomial algebra over one generator $\alpha_i$.

\subsection{Definition of $\ci$}
We define $\ci$ as the DG category of {\em one-sided twisted complexes} over $\ci'$ following \cite{BK}.
More precisely, an object of $\ci$ is of the form:
$$\{(\bigoplus_{i=1}^{n}\mf{a}_i, f=\sum_{i < j}f^j_i) ~|~ f^j_i \in \Hom^1_{\ci'}(\mf{a}_i,\mf{a}_j), ~\sum_k (f^j_k \circ f^k_i)=0 ~\mbox{for all}~ i,j\},$$
where $\mf{a}_i$ are objects of $\ci'$ for $1 \leq i \leq n$.

\begin{rmk}
The notion of {\em one-sided twisted complex} was introduced in \cite[Section 4, Definition 1]{BK}.
The condition in general is $\sum_k (f^j_k \circ f^k_i) + d_{\ci'}(f^j_i)=0$.
But the differential $d_{\ci'}$ is zero in our case so that we have the simplified condition as above.
The term ``one-sided" refers to the condition $f^j_i=0$ for $i\geq j$.
\end{rmk}

The morphism set $\Hom_{\ci}((\bigoplus\mf{a}_i, f=\sum f^{i'}_i), (\bigoplus\mf{b}_j, g=\sum g^{j'}_j))$ is an $\F$-vector space spanned by $\{h_i^j \in \Hom_{\ci'}(\mf{a}_i, \mf{b}_j)\}$ with a differential $d_{\ci}$:
$$d_{\ci}(h_i^j)=\sum_{i'} (h_i^j \circ f_{i'}^i) + \sum_{j'} (g^{j'}_j \circ h^j_i).$$

\begin{rmk}
The reason of enlarging $\ci'$ to $\ci$ is that the $0$th cohomology category $H(\ci)$ of $\ci$ is a triangulated category by \cite[Section 4, Proposition 2]{BK}.
In other words, we can define a cone of a morphism in $\ci$, but not in $\ci'$.
\end{rmk}

The monoidal structure on $\ci'$ extends to $\ci$ in the usual way.
We finally have a strict monoidal DG category $\ci$.

\vspace{.2cm}
\n $\bullet$ {\bf The Grothendieck group $K_0(\ci)$:} we refer to \cite{Ke} for an introduction to DG categories and their homology categories.
Let $H(\ci)$ denote the $0$th cohomology category of $\ci$, which is a triangulated category by \cite[Section 4, Proposition 2]{BK}.
Let $K_0(\ci)$ denote the Grothendieck group of $H(\ci)$.
The induced monoidal functor $\ot: H(\ci) \times H(\ci) \ra H(\ci)$ is bi-exact and hence descends to a multiplication $K_0(\ot): K_0(\ci) \times K_0(\ci) \ra K_0(\ci)$.

Let $[Q] \in K_0(\ci)$ denote the class of the generating object $Q$ of $\ci$.
Then $K_0(\ci)$ is a $\Z$-algebra generated by $[Q]$ with unit $[Q^0]$.

\begin{proof}[Proof of Proposition \ref{K0ci}]
The isomorphism $Q^2 \cong Q^0[1]$ in $\ci$ descends to $[Q]^2=-1$ in $K_0(\ci)$.
Therefore we have a ring homomorphism:
$$\begin{array}{cccc}
\gamma: & \Z[\im] & \ra & K_0(\ci) \\
& \im & \mapsto & [Q],
\end{array}$$
which is surjective.
\end{proof}
We will show that $\gamma$ is actually an isomorphism in the next section.

\section{A categorical action of $\ci$ on $DGP(R)$}
In Section 3.1, we define the DG category $DGP(R)$ which is generated by two projective DG $R$-modules and show that $K_0(DGP(R)) \cong \Z^2$.
In Section 3.2, we construct a DG $R$-bimodule $M$ and show that tensoring with $M$ over $R$ gives an endorfunctor of $DGP(R)$.
Moreover, we prove that $M \ot_R M \cong R[1]$ as $R$-bimodules.
In Section 3.3, we construct a categorical action of $\ci$ on $DGP(R)$, where the object $Q$ of $\ci$ acts on $DGP(R)$ by tensoring with $M$.
Then we show that the induced action on the Grothendieck groups is isomorphic to the action of $\Z[\im]$ on $\Z^2$.

\subsection{The category $DGP(R)$}
We define the DG $\F$-algebra $R$ as a quotient of a quiver algebra of the Khovanov-Seidel quiver with two vertices.
$$\xymatrix{
x \ar@/_1pc/[r]^{a}  & y \ar@/_1pc/[l]_{b}
}
$$

\begin{defn}
Let $R$ be a DG $\F$-algebra with two idempotents $e(x), e(y)$, two generators $a, b$ and relations:
\begin{gather*}
e(z)  e(z')=\delta_{z,z'}e(z) ~\mbox{for}~ z,z' \in \{x,y\};\\
e(x)  a =a  e(y)=a;\\
e(y)  b =b  e(x)=b;\\
a  b  a=b  a  b=0.
\end{gather*}
The differential on $R$ is trivial.
The cohomological grading is zero except $\op{gr}(b)=1$.
\end{defn}

We refer to \cite[Section 10]{BL} for an introduction to DG modules and {\em projective} DG modules.
Let $P(x)=R e(x)$ and $P(y)=R e(y)$ denote two projective DG $R$-modules.
Let $DG(R)$ denote the DG category of DG left $R$-modules.

We define $DGP(R)$ as the smallest full subcategory of $DG(R)$ which contains the projective DG $R$-modules $\{P(x), P(y)\}$ and is closed under the cohomological grading shift functor $[1]$ and taking mapping cones.
The $0$th cohomology category $H(DGP(R))$ is just the homotopy category of bounded complexes of projective modules $\{P(x), P(y)\}$.

\begin{lemma} \label{K0R}
The Grothendieck group $K_0(DGP(R))$ of $H(DGP(R))$ is isomorphic to $Z\lan x,y \ran$.
\end{lemma}
\begin{proof}
Since $H(DGP(R))$ is generated by $P(x)$ and $P(y)$, the map
$$\begin{array}{ccc}
\Z\lan x,y \ran & \ra & K_0(DGP(R)) \\
 x & \mapsto & [P(x)], \\
 y & \mapsto & [P(y)],
\end{array}$$
is surjective.
It is also injective because the dimension vectors of $P(x)$ and $P(y)$ are $(2,1)$ and $(1,2)$ which are linearly independent.
\end{proof}

\subsection{The DG $R$-bimodule $M$}
As left $R$-modules, we define
$$M=P(y) \bigoplus P(x)[1],$$
where $_ym_x \in P(y)$ and $ _xm_y \in P(x)[1]$ are the generators with $\op{gr}(_ym_x)=0, \op{gr}(_xm_y)=-1$.
In other words, $M$ is a $6$-dimensional $\F$-vector space with a basis:
$$\{_ym_x,~ a\cdot{}_ym_x,~ ba\cdot{}_ym_x,~ _xm_y,~ b\cdot{}_xm_y,~ ab\cdot{}_xm_y\}.$$
The right $R$-module structure on $M$ is given by:
$$_ym_x\cdot e(x)={}_ym_x, \quad _xm_y\cdot e(y)={}_xm_y, \quad {}_ym_x\cdot a=b\cdot{}_xm_y, \quad _xm_y\cdot b=a\cdot{}_ym_x.$$
The differential on $M$ is trivial.
It is easy to verify that $M$ is a well-defined $R$-bimodule.

\begin{lemma}  \label{tensor}
We have tensor products $M \ot_R P(x)\cong P(y), M \ot_R P(y)\cong P(x)[1]$ as DG left $R$-modules.
Hence, tensoring with $M$ is an endorfunctor of $DGP(R)$.
\end{lemma}
\begin{proof}
We directly compute that $M \ot_R P(x)$ has a basis
$$\{_ym_x \ot e(x),~ a\cdot{}_ym_x \ot e(x),~ ba\cdot{}_ym_x \ot e(x)\}$$
so that it is isomorphic to $P(y)$ as left $R$-modules.
The other isomorphism is similar.
Hence tensoring with $M$ preserves $DGP(R)$ since $DGP(R)$ is generated by $P(x), P(y)$.
\end{proof}

Comparing to the isomorphism $Q^2\cong Q^0[1]$ in $\ci$, we have the following isomorphisms of DG $R$-bimodules.

\begin{lemma}\label{isoM}
The tensor product $M\ot_R M$ is isomorphic to $R[1]$ as DG $R$-bimodules.
\end{lemma}
\begin{proof}
We define the following map of $R$-bimodules
$$\begin{array}{cccc}
f: & R[1] & \ra & M\ot_R M \\
 & e(x) & \mapsto & {}_xm_y\ot{}_ym_x, \\
 & e(y) & \mapsto & {}_ym_x \ot{}_xm_y.
\end{array}$$
This is well-defined since
\begin{gather*}
f(be(x))=b \cdot {}_xm_y\ot{}_ym_x={}_ym_x \cdot a\ot{}_ym_x={}_ym_x \ot a \cdot {}_ym_x={}_ym_x\ot{}_xm_y\cdot b=f(e(y)b), \\
f(ae(y))=a \cdot {}_ym_x\ot{}_xm_y={}_xm_y \cdot b\ot{}_xm_y={}_xm_y \ot b \cdot {}_xm_y={}_xm_y\ot{}_ym_x\cdot a=f(e(x)a).
\end{gather*}
The gradings $\op{gr}({}_xm_y\ot{}_ym_x)=\op{gr}({}_ym_x \ot{}_xm_y)=-1$ agree with the gradings $\op{gr}(e(x))=\op{gr}(e(y))=-1$ in $R[1]$.
It is easy to see that this map is actually an isomorphism.
\end{proof}

\subsection{The action of $\ci$ on $DGP(R)$}
We define a functor $\tau: \ci \ra DG(R\ot_{\F} R^{op})$ of monoidal DG categories by $\tau(Q^0)=R$ and $\tau(Q)=M$ on the objects.
Here $\tau$ maps the monoidal structure on $\ci$ to the monoidal structure on $DG(R\ot_{\F} R^{op})$ given by tensoring $R$-bimodules over $R$.
From now on, we use $\Hom_{R}$ to denote Hom spaces in $DG(R\ot_{\F} R^{op})$ for simplicity.

\vspace{.2cm}
\n $\bullet$ {\bf Definition of $\tau$ on morphisms.}
There are $3$ nontrivial generators $cup, cap, hf$ in $\ci$.

Recall that $cup \in \Hom_{\ci}(Q^0, Q^2)$ and $cap \in \Hom_{\ci}(Q^2, Q^0)$ give the isomorphism $Q^2 \cong Q^0[1]$.
We define $$\tau(cup)=f, \quad \tau(cap)=f^{-1},$$
where $f \in \Hom_R(R, M\ot M)$ is given in the proof of Lemma \ref{isoM}.

\vspace{.2cm}
For the half strand $hf \in \Hom_{\ci}(Q^0, Q)$, we define $\tau(hf)=g \in \Hom_R(R, M)$ as
$$\begin{array}{cccc}
g: & R & \ra & M \\
 & e(x) & \mapsto & {}_xm_y\cdot b=a\cdot{}_ym_x, \\
 & e(y) & \mapsto & {}_ym_x\cdot a=b\cdot{}_xm_y.
\end{array}$$
We check that $g$ is a well-defined $R$-bimodule map:
\begin{gather*}
g(be(x))=b \cdot {}_xm_y\cdot b = {}_ym_x\cdot ab = g(e(y)b), \\
g(ae(y))=a \cdot {}_ym_x\cdot a = {}_xm_y\cdot ba = g(e(x)a).
\end{gather*}

\vspace{.2cm}
\n $\bullet$ {\bf Verification of $\tau$ on the relations.}
By definition $\tau$ maps the isotopy relation (R1-c) to identity homomorphisms of $R$-bimodules.
The relations (R2) and (R3) are preserved under $\tau$ because $\tau(cup)=f: R[1] \ra M \ot M$ is the isomorphism.
It remains to verify the relations (R1-b) and (R4).

\vspace{.2cm}
For (R1-b), it is enough to show that
$$(\tau(cap) \ot \tau(id_Q)) \circ (\tau(id_Q) \ot \tau(cup))=(f^{-1} \ot id_M) \circ (id_M \ot f) = id_M \in \Hom_R(M, M).$$
We check this on the generator ${}_xm_y$:
\begin{align*}
(f^{-1} \ot id_M) \circ (id_M \ot f)({}_xm_y)&=(f^{-1} \ot id_M) \circ (id_M \ot f)({}_xm_y \ot e(y)) \\
&=(f^{-1} \ot id_M)({}_xm_y \ot {}_ym_x \ot {}_xm_y) \\
&=e(x)\cdot {}_xm_y = {}_xm_y.
\end{align*}
The verification on the other generator ${}_ym_x$ is similar.

\vspace{.2cm}
For (R4), it is enough to show that
$$\tau(hf) \ot \tau(id_Q) = g \ot id_M = id_M \ot g = \tau(id_Q) \ot \tau(hf) \in \Hom_R(M, M \ot M).$$
We check this on the generator ${}_xm_y$:
\begin{align*}
(g \ot id_M)({}_xm_y)=&(g \ot id_M)(e(x) \ot {}_xm_y) ={}_xm_y\cdot b \ot {}_xm_y= {}_xm_y \ot b \cdot {}_xm_y\\
=&{}_xm_y \ot {}_ym_x \cdot a = (id_M \ot g)({}_xm_y \ot e(y)) = (id_M \ot g)({}_xm_y).
\end{align*}
The verification on the other generator ${}_ym_x$ is similar.

\vspace{.2cm}
As a conclusion, the functor $\tau: \ci \ra DG(R\ot R^{op})$ is well-defined.
Since tensoring with $\tau(Q)=M$ is an endorfunctor of $DGP(R)$, $\tau$ induces a categorical action $\eta: \ci \times DGP(R) \ra DGP(R)$ via tensoring with $R$-bimodules.

\vspace{.2cm}
\n $\bullet$ {\bf Computation of $K_0(\eta)$.}
Let $K_0(\eta): K_0(\ci) \times \Z\lan x,y \ran \ra \Z\lan x,y \ran$ be the induced map on the Grothendieck groups under the isomorphism $K_0(DGP(R)) \cong \Z\lan x,y \ran$ in Lemma \ref{K0R}.
Recall from Lemma \ref{K0R} that $\gamma: \Z[\im] \ra K_0(\ci)$ is surjective.
We compute the pullback of $K_0(\eta)$ under $\gamma$ by Lemma \ref{tensor}:
\begin{gather*}
\im \cdot x = [Q] \cdot [P(x)] =[M \ot_R P(x)]=[P(y)]=y, \\
\im \cdot y = [Q] \cdot [P(y)] =[M \ot_R P(y)]=[P(x)[1]]=-x.
\end{gather*}
The pullback map agrees with the action of $\Z[\im]$ on $\Z\lan x,y \ran$.

Since this action is faithful, we conclude that $\gamma$ is an injective map, hence an isomorphism.
Therefore, we finish proving Theorems \ref{thm1} and \ref{thm2}.

\end{document}